\documentclass[10pt]{scrartcl}
\usepackage{tikz}
\usetikzlibrary{matrix,arrows,decorations.pathmorphing,shapes.geometric}
\usepackage{etex}
\usepackage{subcaption}
\usepackage[utf8]{inputenc}
\usepackage{a4wide}
\usepackage{graphicx}
\usepackage{wrapfig}
\usepackage[hmarginratio={1:1},     
vmarginratio={1:1},     
textwidth=15cm,        
textheight=22cm,      
heightrounded,]{geometry}
\usepackage{amsmath,accents}
\usepackage{amsthm}
\usepackage{amssymb}

\usepackage{mathrsfs, mathtools}
\usepackage[T1]{fontenc}
\usepackage{ae,aecompl}
\usepackage[all,cmtip]{xy}

\makeatletter
\DeclareRobustCommand{\em}{%
	\@nomath\em \if b\expandafter\@car\f@series\@nil
	\normalfont \else \slshape \fi}
\makeatother

\usepackage{tikz}
\usetikzlibrary{matrix,arrows,decorations.pathmorphing,shapes.geometric}
\usepackage{tikz-cd}
\usepackage{enumitem}

\usepackage[hidelinks]{hyperref}

\numberwithin{equation}{section}
\usepackage{mathtools}
\mathtoolsset{showonlyrefs}
\numberwithin{equation}{section}

\newtheoremstyle{style1}
{13pt}
{13pt}
{}
{}
{\normalfont\bfseries}
{.}
{.5em}
{}

\theoremstyle{style1}

\newtheorem{definition}{Definition}[section]
\newtheorem{example}[definition]{Example}
\newtheorem{remark}[definition]{Remark}

\makeatletter
\newtheorem*{repd@theorem}{\repd@title}
\newcommand{\newrepdtheorem}[2]{%
	\newenvironment{repd#1}[1]{%
		\def\repd@title{#2 \ref{##1}}%
		\begin{repd@theorem}}%
		{\end{repd@theorem}}}
\makeatother

\newrepdtheorem{definition}{Definition}

\newcommand{\catf}[1]{{\mathsf{#1}}}

\newtheoremstyle{style2}
{13pt}
{13pt}
{\slshape}
{}
{\normalfont\bfseries}
{.}
{.5em}
{}

\theoremstyle{style2}

\makeatletter
\newtheorem*{rep@theorem}{\rep@title}
\newcommand{\newreptheorem}[2]{%
	\newenvironment{rep#1}[1]{%
		\def\rep@title{#2 \ref{##1}}%
		\begin{rep@theorem}}%
		{\end{rep@theorem}}}
\makeatother

\newreptheorem{theorem}{Theorem}
\newreptheorem{corollary}{Corollary}

\newtheorem{lemma}[definition]{Lemma}
\newtheorem{theorem}[definition]{Theorem}

\newtheorem{corollary}[definition]{Corollary}

\newcommand{\Z}{\mathbb{Z}}
\newcommand{\R}{\mathbb{R}}

\newcommand{\Map}{\catf{Map}}
\newcommand{\Diff}{\catf{Diff}}
\newcommand{\cat}[1]{\mathcal{#1}}
\newcommand{\ra}[1]{\xrightarrow{\ #1 \ }}
\let\DH\undefined
\newcommand{\DH}{\catf{DH}}
\usepackage{needspace}
\newcommand{\spaceplease}{\needspace{5\baselineskip}}

\newcommand{\Ch}{\catf{Ch}_k}
\newcommand{\hocolim}{\operatorname{hocolim}}

\newcommand{\Hom}{\operatorname{Hom}}

\newcommand{\id}{\operatorname{id}}
\newcommand{\Cat}{\catf{Cat}}
\newcommand{\DS}{\text{/\hspace{-0.1cm}/}}
\let\to\undefined
\newcommand{\to}{\longrightarrow}
\let\mapsto\undefined
\newcommand{\mapsto}{\longmapsto}

\newcommand{\Lexf}{\catf{Lex}^\mathsf{f}}

\newcommand{\opp}{\text{opp}}
\newcommand{\horb}{h \, }
\newcommand{\Surfc}{\catf{Surf}^{\catf{c}}}

\newcommand{\lint}{\int_{\mathbb{L}}}
\newcommand{\Proj}{\catf{Proj}\,}
\DeclareMathSymbol{\Phiit}{\mathalpha}{letters}{"08} 
\DeclareMathSymbol{\Psiit}{\mathalpha}{letters}{"09}
\DeclareMathSymbol{\Sigmait}{\mathalpha}{letters}{"06}
\DeclareMathSymbol{\Xiit}{\mathalpha}{letters}{"04}
\DeclareMathSymbol{\Piit}{\mathalpha}{letters}{"05}\let\Pi\undefined\newcommand{\Pi}{\Piit}
\DeclareMathSymbol{\Gammait}{\mathalpha}{letters}{"00}
\DeclareMathSymbol{\Omegait}{\mathalpha}{letters}{"0A}\let\Omega\undefined\newcommand{\Omega}{\Omegait}
\DeclareMathSymbol{\Upsilonit}{\mathalpha}{letters}{"07}
\DeclareMathSymbol{\Thetait}{\mathalpha}{letters}{"02}
\DeclareMathSymbol{\Lambdait}{\mathalpha}{letters}{"03}\let\Lambda\undefined\newcommand{\Lambda}{\Lambdait}
\let\Phi\undefined\newcommand{\Phi}{\Phiit}
\let\Sigma\undefined\newcommand{\Sigma}{\Sigmait}
\let\Psi\undefined\newcommand{\Psi}{\Psiit}
\let\Gamma\undefined\newcommand{\Gamma}{\Gammait}

\newenvironment{pnum}{\begin{enumerate}[label=(\roman*)]}{\end{enumerate}}

\setkomafont{section}{\rmfamily\large}
\setkomafont{subsection}{\rmfamily}
\setkomafont{subsubsection}{\rmfamily}
\setkomafont{subparagraph}{\rmfamily} 

\makeatletter
\renewcommand\section{\@startsection {section}{1}{\z@}%
	{-3.5ex \@plus -1ex \@minus -.2ex}%
	{2.3ex \@plus.2ex}%
	{\normalfont\scshape\centering}}
\makeatother

\usepackage{titletoc}

\dottedcontents{section}[1em]{\rmfamily}{1em}{0.2cm}
\dottedcontents{subsection}[0em]{}{3.3em}{1pc}

\begin{document}

	\vspace*{-1.5cm}	\begin{flushright}\small		{\sffamily CPH-GEOTOP-DNRF151} \\	\textsf{January 2022}	\end{flushright}	\vspace{5mm}	\begin{center}	\textbf{\Large{The Diffeomorphism Group of the Solid Closed Torus	 \\[0.5ex]
				and Hochschild Homology}}\\	\vspace{1cm}	{\large Lukas Müller $^{a}$} \ \ and \ \ {\large Lukas Woike $^{b}$}\\ 	\vspace{5mm}{\slshape $^a$ Max-Planck-Institut f\"ur Mathematik\\ Vivatsgasse 7 \\  D-53111 Bonn}\\ \emph{lmueller4@mpim-bonn.mpg.de }	\\[7pt]	{\slshape $^b$ Institut for Matematiske Fag\\ K\o benhavns Universitet\\	Universitetsparken 5 \\  DK-2100 K\o benhavn \O }\\ \ \emph{ljw@math.ku.dk }\end{center}	\vspace{0.3cm}	
	\begin{abstract}\noindent 
		We prove that for a self-injective ribbon Grothendieck-Verdier category $\cat{C}$ in the sense of Boyarchenko-Drinfeld the cyclic action on the Hochschild complex of $\cat{C}$ extends to an action of the diffeomorphism group of the solid closed torus $\mathbb{S}^1 \times \mathbb{D}^2$.	
	\end{abstract}

\tableofcontents

\section{Introduction and summary}
The Hochschild complex of any associative algebra $A$ in a suitable (higher) symmetric monoidal category can be defined as the homotopy colimit of the simplicial object
	\begin{equation}\label{eqnHochschildobject}
	\begin{tikzcd}
		\dots \ar[r, shift left=6]  \ar[r, shift left=2]
		\ar[r, shift right=6]  \ar[r, shift right=2]
		& 	\displaystyle A^{\otimes 3}
		\ar[l, shift left=4]  \ar[l]
		\ar[l, shift right=4]  
		\ar[r, shift left=4] \ar[r, shift right=4] \ar[r] & \displaystyle A^{\otimes 2} \ar[r, shift left=2] \ar[r, shift right=2]
		\ar[l, shift left=2] \ar[l, shift right=2]
		& \displaystyle A\ ,  \ar[l] 
	\end{tikzcd}
\end{equation}
whose face maps use the product of $A$ while the degeneracy maps insert units.
For an algebra in vector spaces, the homology groups of the chain complex 
corresponding to~\eqref{eqnHochschildobject}
are the familiar Hochschild homology groups.
The cyclic permutation of the tensor copies of $A$ in each degree induces a cyclic symmetry, i.e.\ 
 an action of the topological group $\mathbb{S}^1$. The homotopy orbits of this action are known as \emph{cyclic homology}. The study of cyclic homology was initiated in the 1980s independently by Connes \cite{connes} and Tsygan \cite{tsygan}. 
 Of course, instead of considering the Hochschild complex of an algebra, we can consider the Hochschild complex of a linear or a differential graded category. 
 
The fact that one encounters 
in almost all areas of mathematics
\emph{structured} algebras naturally leads to the following question:
\begin{itemize}
	\item[(Q)]\emph{What kind of additional structure does one need on an algebra such that the cyclic symmetry can be extended in a meaningful way, for example to an action of a much larger topological group?}
	\end{itemize}
This is obviously an open-ended question to which certainly a lot of different answers can be given. For example, if $A$ comes equipped with an involution through anti algebra maps, then one finds an action of $\text{O}(2)=\mathbb{S}^1 \rtimes \Z_2$	
on the Hochschild complex of $A$. In the semidirect product $\mathbb{S}^1 \rtimes \Z_2$, the cyclic group $\Z_2$ with two elements acts on $\mathbb{S}^1$ by reflection.
This extends the theory of cyclic homology and leads to	 \emph{dihedral homology} as introduced by Loday~\cite{loday}.
A very systematic approach to the Hochschild homology of structured algebras is given by Wahl and Westerland in \cite{wahlwesterland,wahl}.

This short article offers a different and, at least at first sight, surprising answer to question~(Q) that applies to an important structure in quantum algebra, namely \emph{ribbon Grothendieck-Verdier structures} introduced by Boyarchenko and Drinfeld \cite{bd} based on Barr's notion of a $\star$-autonomous category \cite{barr}. 
Roughly, 
a Grothendieck-Verdier category is a monoidal category $\cat{C}$ equipped with an equivalence $D:\cat{C}\to\cat{C}^\opp$
called 
\emph{duality functor} compatible in a specific way with the monoidal structure. A ribbon Grothendieck-Verdier structure is additionally equipped with
  a \emph{braiding} (natural isomorphisms $c_{X,Y}: X\otimes Y \to Y\otimes X$ compatible with the monoidal unit and subject to the hexagon axioms) 
and a \emph{balancing} (a natural automorphism $\theta_X : X \to X$
such that $\theta_I=\id_I$ for the monoidal unit $I$ and $\theta_{X\otimes Y}=c_{Y,X}c_{X,Y}(\theta_X\otimes\theta_Y)$ for $X,Y\in\cat{C}$), that additionally satisfies $\theta_{DX}=D\theta_{X}$, see Section~\ref{secgv}.

Grothendieck-Verdier duality allows us to generalize the notion of a \emph{finite tensor category} \cite{etingofostrik} which includes \emph{rigidity} (existence of left/right duals) as part of its definition, i.e.\ it requires that every object $X$ has a dual
$X^\vee$ that comes with an evaluation $X^\vee \otimes X \to I$ and a coevaluation $I\to X\otimes X^\vee$ subject to the so-called zigzag identities. 
The notion of a finite tensor category 
  is the backbone of the approach to quantum algebra laid out in the monograph \cite{egno} by Etingof, Gelaki, Ostrik and Nikshych. Every rigid monoidal category can be seen as a Grothendieck-Verdier category whose duality functor sends an object to its dual,  but not all Grothendieck-Verdier categories are of this form. 

We will be interested in Grothendieck-Verdier categories in a \emph{$k$-linear setting},
where $k$ is a fixed algebraically closed field, more precisely in ribbon Grothendieck-Verdier categories in the symmetric monoidal bicategory $\Lexf$ of $k$-linear abelian categories (subject to some finiteness conditions), left exact functors and natural transformations. 
Moreover, we will require \emph{self-injectivity}, i.e.\ the projective objects must coincide with the injective ones.

Our main result is concerned with the Hochschild complex of a self-injective ribbon Grothendieck-Verdier category $\cat{C}$ in $\Lexf$. The Hochschild complex of $\cat{C}$ is, as usual, the realization of the simplicial vector space
		\begin{equation}
			\begin{tikzcd}
				\dots
				\ar[r, shift left=4] \ar[r, shift right=4] \ar[r] & \displaystyle \bigoplus_{X_0,X_1 \in \Proj\cat{C}}  \cat{C}(X_1,X_0)\otimes \cat{C}(X_0,X_1) \ar[r, shift left=2] \ar[r, shift right=2]
				\ar[l, shift left=2] \ar[l, shift right=2]
				& \displaystyle \bigoplus_{X_0 \in \Proj \cat{C}} \cat{C}(X_0,X_0)\ .  \ar[l] 
			\end{tikzcd}
	\end{equation}
The fact that just the projective objects are used to define the complex is standard in this context. It ensures that if $\cat{C}$ is given, as a linear category, by finite-dimensional modules over a finite-dimensional algebra $A$, the above Hochschild complex is actually equivalent to the `standard one' of $A$ by the so-called \emph{Agreement Principle} \cite{mcarthy,keller}. 
One may see the Hochschild complex as the homotopy coend $\lint^{X\in\Proj\cat{C}}\cat{C}(X,X)$ over the endomorphism spaces of projective objects, and we will use this as our notation for the Hochschild complex. 
We may now state our main result:

	\begin{reptheorem}{mainthm}
	Let $\cat{C}$ be a self-injective ribbon Grothendieck-Verdier category in $\Lexf$.
	Then its duality functor and its balancing induce on the Hochschild complex $\lint^{X\in\Proj\cat{C}}\cat{C}(X,X)$
	an action of the diffeomorphism group $\Diff(\mathbb{S}^1\times\mathbb{D}^2)$ of the solid closed torus that extends the usual cyclic symmetry of the Hochschild complex. 
\end{reptheorem}

Here and elsewhere in the article, diffeomorphisms and mapping classes will always be orientation-preserving. Moreover, a group action on an object in a higher category, such as chain complexes, has to be understood as a \emph{homotopy coherent} action.

We prove in Corollary~\ref{cordep} that generally the $\Diff(\mathbb{S}^1\times\mathbb{D}^2)$-action
 depends on the ribbon Grothen\-dieck-Verdier structure, i.e.\ in contrast to the Hochschild complex and its cyclic action, it is sensitive to more than the linear structure.
 
 Let us highlight concrete situations to which Theorem~\ref{mainthm} applies:
 
 \begin{itemize}
 	
 	\item We had mentioned above that the result applies in particular to \emph{finite ribbon categories} in the sense of \cite{egno} --- by definition these categories are always rigid. A finite ribbon category can be obtained, for example, by taking finite-dimensional modules over a finite-dimensional ribbon Hopf algebra~\cite[XIV.6]{kassel}. In that case, the Hochschild complex $\lint^{X\in\Proj\cat{C}}\cat{C}(X,X)$ is equivalent to the Hochschild complex of $A$. This implies that the Hochschild complex of a finite-dimensional ribbon Hopf algebra comes with a $\Diff(\mathbb{S}^1\times\mathbb{D}^2)$-action, see Example~\ref{exhopf} for more details.
 	
 	\item While the case of finite ribbon categories and in particular ribbon Hopf algebras exhibits a rich class of examples, it is important that the notion of Grothendieck-Verdier duality is designed to go beyond rigid monoidal categories. For example, Allen, Lentner, Schweigert and Wood prove in \cite{alsw} that suitable choices of modules over a vertex operator algebra lead to ribbon Grothendieck-Verdier categories
 	(this includes categories with a non-exact monoidal product which therefore cannot be rigid). To these categories, Theorem~\ref{mainthm} may also be applied if they are self-injective.

 	\end{itemize}

 Using our main result we can also exhibit at least one instance in which there is a non-trivial action of the \emph{higher} homotopy groups of diffeomorphism groups on so-called \emph{differential graded conformal blocks}, see Remark~\ref{dmf}. To the best of our knowledge, this is a phenomenon that has not been observed so far.

	\vspace*{0.2cm}\textsc{Acknowledgments.} We are grateful to Andrea Bianchi, Adrien Brochier, S\o ren Galatius,		
	Christoph Schweigert,
Nathalie Wahl and Simon Wood
for  helpful discussions related to this project.
Additionally, we thank the anonymous referee for helpful suggestions.
LM gratefully acknowledges support by the Max Planck Institute for Mathematics in Bonn.
LW gratefully acknowledges support by 
the Danish National Research Foundation through the Copenhagen Centre for Geometry
and Topology (DNRF151)
and  by the European Research Council (ERC) under the European Union's Horizon 2020 research and innovation programme (grant agreement No.~772960).

\section{The diffeomorphism group of $\mathbb{S}^1 \times \mathbb{D}^2$ via dihedral homology}
In this section, we give a specific model for the diffeomorphism group of $\mathbb{S}^1\times\mathbb{D}^2$ using dihedral homology. 
This will in the next section allow us to write its classifying space $B\Diff (\mathbb{S}^1\times\mathbb{D}^2)$ in a way that is adapted to the algebraic structure of a self-injective ribbon Grothendieck-Verdier category. 

Let us first recall some well-known facts about diffeomorphism groups of handlebodies. Throughout the article,
handlebodies will be three-dimensional, and diffeomorphisms and their mapping classes will be orientation-preserving.
Since any handlebody $H$ is a Haken manifold,
the restriction
$\Diff_0(H)\to \Diff_0(\partial H)$
is a fibration with contractible fiber by~\cite[Theorem~2]{hatcher} (by $\Diff_0$ we denote the identity component of $\Diff$). 
Unless $H$ is $H_{0,0}=\mathbb{B}^3$ or $H_{1,0}=\mathbb{S}^1\times\mathbb{D}^2$ (we use here $H_{g,n}$ to denote the handlebody of genus $g$ and $n$ embedded disks), this proves that $\Diff(H)$ is homotopy discrete, i.e.\ that the map $\Diff(H)\to\Map(H)$ is a homotopy equivalence (because $\Diff(\partial H)$ is homotopy discrete in these cases \cite{earle-eells}). 
It also tells us 
$\Diff_0(H_{1,0})\simeq \mathbb{T}^2$ thanks to $\Diff_0(\mathbb{T}^2)\simeq \mathbb{T}^2$ \cite[Théorème~1]{gramain}.
Next recall that 
we have an exact sequence
\begin{align} 0 \to \Diff_0(H_{1,0})\to\Diff(H_{1,0})\to \Diff(H_{1,0})/\Diff_0(H_{1,0})\to 0 \ ,\label{eqnses}
\end{align}
where  $\Diff(H_{1,0})/\Diff_0(H_{1,0})= \Map(H_{1,0})$ is the mapping class group of $H_{1,0}$. 
The mapping class group $\Map(H_{1,0})\cong \Z\times \Z_2$ is generated by a Dehn twist $T$ along any properly embedded disk in $H_{1,0}$ and the rotation $R$ by $\pi$ around any axis in the plane in which the torus lies~\cite[Theorem~14]{Wajnryb}. \label{labeldescriptioTandR}
Under the isomorphism $\Map(\mathbb{T}^2)\cong \text{SL}(2,\Z)$
sending a mapping class to the induced automorphism on the first homology of $\mathbb{T}^2$, the inclusion $\Map(H_{1,0})\subset \Map(\mathbb{T}^2)$ sends 
\begin{align}
\label{eqnTandRmatrix}	T \mapsto \begin{pmatrix} 1 & 0 \\ 1 & 1 \end{pmatrix}\ , \quad R \mapsto \begin{pmatrix} -1 & \phantom{-}0\\\phantom{-}0&-1 \end{pmatrix} \ . \end{align}
Through the matrix representation~\eqref{eqnTandRmatrix}
of $\Z\times\Z_2$, we have a section of the epimorphism in~\eqref{eqnses}. This gives us
$\Diff(H_{1,0})\cong \Diff_0(H_{1,0})\rtimes \Map(H_{1,0})$. The connected component $\Diff_0(H_{1,0})$ is homotopy equivalent to a torus $\mathbb{T}^2$ as just explained.
The action of $\Map(H_{1,0})$ on $\mathbb{T}^2$ sees $\Map(H_{1,0})$ as subgroup of $\text{SL}(2,\Z)$ which acts on $\mathbb{R}^2 / \Z^2=\mathbb{T}^2$.

In order to present our combinatorial model for $\Diff(H_{1,0})$, we also need to recall the notion of dihedral homology:
Recall that \emph{Connes' cyclic category $\Lambda$} \cite{connes} is the category with objects $\mathbb{N}_0$; we denote the object corresponding to $n\ge 0$ by $[n]$.
A morphism $f: [n]\to [m]$ is given by an
equivalence class of functions $f: \Z \to \Z $ such that $f(i+n+1)=f(i)+m+1$ modulo the relation $f\sim g$ if $f-g$ is a constant multiple of $m+1$.  
The category $\Lambda$ contains the simplex category $\Delta$ as subcategory.
As generating morphisms, it has the face and degeneracy maps that we already know from the simplex category $\Delta$
and the \emph{cyclic permutations}
	$\tau_n : 
	[n] \to [n]$ represented by maps $\Z\to\Z$ that shift by one.
The cyclic permutations fulfill, besides the obvious relation $\tau_n^{n+1} = \id_{[n]}$, further compatibility relations with the face and degeneracy maps. 
We denote by $\vec\Lambda \subset \Lambda$ 
the subcategory of the cyclic category without degeneracy maps. 
This subcategory inclusion is homotopy initial, thereby making $\vec\Lambda^\opp \subset \Lambda^\opp$ homotopy final. We use here the terminology of \cite[Chapter~8.5]{riehl}.

The category $\Lambda$ has a natural action of $\Z_2$ through the \emph{reversal 
	functor} $r: \Lambda \to \Lambda$ which is the identity on objects and sends any morphism $f:[n]\to[m]$ 
in $\Lambda$ to
$r(f):[n]\to [m]$ given by $(r(f)) (p):=   m- f(n-p)$.
We denote by $\Lambda \rtimes \Z_2$ the Grothendieck construction of the functor
$
*\DS \Z_2 \to \Cat$
from the groupoid with one object and automorphism group $\Z_2$ to the category $\Cat$ of categories
sending $*$ to $\Lambda$ and the generator $-1\in\Z_2$ to the reversal functor $r:\Lambda \to \Lambda$
(recall that the Grothendieck construction $\int F$
of a functor $F:\cat{C}\to\Cat$ is the category of pairs $(c,x)$ formed by all $c\in\cat{C}$ and $x\in F(c)$, see e.g.~\cite[Section~I.5]{maclanemoerdijk}).
The category $\Lambda \rtimes \Z_2$ can be identified with the \emph{dihedral category} \cite{loday,spalinski}. 
Restriction to $\vec\Lambda$ yields a functor $\vec r : \vec\Lambda \to \vec\Lambda$. 
This allows us to define $\vec\Lambda \rtimes\Z_2$, the \emph{semidihedral category}, also via a Grothendieck construction. Functors out of the opposite categories of $\Lambda \rtimes \Z_2$ and $\vec\Lambda \rtimes \Z_2$ are called \emph{dihedral objects} and \emph{semidihedral objects}, respectively.
We can see
$(\vec\Lambda \rtimes \Z_2)^\opp$ as the Grothendieck construction of the $\Z_2$-action on $\vec\Lambda^\opp$ through the functor
${\vec r}\, ^\opp : \vec\Lambda^\opp \to \vec\Lambda^\opp$ induced by $r$. 
In other words, $(\vec\Lambda \rtimes \Z_2)^\opp=\vec\Lambda^\opp \rtimes \Z_2$. Similarly, $(\Lambda \rtimes \Z_2)^\opp = \Lambda^\opp \rtimes \Z_2$.

Given an associative algebra $A$ in a symmetric monoidal $\infty$-category $\cat{S}$,
one can build its Hochschild object, i.e.\ the simplicial object in $\cat{S}$ given in~\eqref{eqnHochschildobject}.
If we assume that $\cat{S}$ is cocomplete, we may take the homotopy colimit of~\eqref{eqnHochschildobject} and obtain the \emph{Hochschild homology} of $A$
(of course, unless the target category is chain complexes, this will not give us homology in the traditional sense of the word, but this extended meaning of the word `homology' is standard).
In fact, the Hochschild object 
is actually a cyclic object through the cyclic permutation of the tensor factors. If we take the homotopy colimit over $\Lambda^\opp$, we obtain the \emph{cyclic homology} of $A$ \cite{connes,tsygan}. If $A$ comes equipped with a $\Z_2$-action through anti algebra maps, then the Hochschild object of $A$ actually extends to a dihedral object. One defines the homotopy colimit over $\Lambda^\opp \rtimes \Z_2$ as the \emph{dihedral homology} of $A$ \cite{loday,spalinski}. We denote it by $\DH(A)$. Thanks to the homotopy finality statements given above, the homotopy colimits involved in the computation of Hochschild homology, cyclic homology and dihedral homology may always be computed \emph{without} degeneracies.

Before stating the next result, let us introduce  further notation: For a space $X$ with $G$-action, we will denote by $X_{\horb G}$ the homotopy orbits of the $G$-action on $X$.
Moreover, we denote by $K(G,n)$ the $n$-th Eilenberg-Mac Lane space for the group $G$ (for $n\ge 2$, this means that $G$ is abelian).

\spaceplease
\begin{lemma}\label{lemmadihedralS1}
	Consider the commutative topological algebra $\mathbb{S}^1$ together with its trivial $\Z_2$-action.
	There is a homotopy equivalence 
	\begin{align}
	\DH(\mathbb{S}^1)    \simeq    \Map (\mathbb{S}^1 , K(\Z,2))_{\horb   \mathbb{S}^1 \rtimes\Z_2   }  \ ,
	\end{align}
	where $\mathbb{S}^1 \rtimes\Z_2 $
	(the $\Z_2$-action on $\mathbb{S}^1$ is by reflection) acts
	on the mapping space  $\Map (\mathbb{S}^1 , K(\Z,2))$ as follows:
	\begin{pnum}
		\item The $\mathbb{S}^1$-action on $\Map (\mathbb{S}^1 , K(\Z,2))$ comes from the natural $\mathbb{S}^1$-action on $\mathbb{S}^1$ and precomposition.
		\label{actionpartone}
		\item The $\Z_2$-action on $\Map (\mathbb{S}^1 , K(\Z,2))$ is given by precomposition
		with a reflection and postcomposition with 
		the map $K(\Z,2)\to K(\Z, 2)$ induced by $-1 : \Z\to\Z$. 
		\label{actionparttwo}
		
	\end{pnum}
\end{lemma}

\begin{proof} 
	Since the geometric realization of $\vec\Lambda^\opp \rtimes \Z_2$ is equivalent to $\mathbb{S}^1 \rtimes \Z_2$ \cite[Proposition~3.11]{loday}, 
	we may describe the dihedral homology of $\mathbb{S}^1$ as the homotopy $\mathbb{S}^1 \rtimes \Z_2$-orbits of the Hochschild homology of $\mathbb{S}^1$. 
	This follows from the $\Z_2$-equivariant analogue of~\cite[Proposition~B.5]{ns}. As a consequence,
	$\DH(\mathbb{S}^1)     \simeq  \left(  \int_{\mathbb{S}^1} \mathbb{S}^1 \right)_{\horb   \mathbb{S}^1 \rtimes\Z_2   }$, where $\int _{\mathbb{S}^1} \mathbb{S}^1$ is the factorization homology of the commutative algebra $\mathbb{S}^1$ evaluated on the circle.
	The  $\mathbb{S}^1\rtimes \Z_2$-action
	on  $\int_{\mathbb{S}^1} \mathbb{S}^1$ can be naturally understood as follows~\cite[Example~2.11]{AF}:
	The space $\mathbb{S}^1=K(\Z,1)$ is an `unoriented topological $E_1$-algebra', i.e.\ an $E_1$-algebra with an anti algebra involution. The involution is
	trivial in this case (the identity of $\mathbb{S}^1 = K(\Z,1)$ is indeed an anti algebra involution since $\Z$ is abelian). This allows us to compute its factorization homology
	over any \emph{unoriented} 1-dimensional manifold. The value of factorization homology $\int_{\mathbb{S}^1}K(\Z,1)$ on $\mathbb{S}^1$ is the Hochschild complex. Hence, the group of not necessarily orientation preserving diffeomorphisms of $\mathbb{S}^1$, which is homotopy equivalent to $\text{O}(2)=\mathbb{S}^1\rtimes \Z_2$, acts on the factorization homology $\int_{\mathbb{S}^1}\mathbb{S}^1$, thereby giving rise to the action appearing in the computation of dihedral homology.

	We have furthermore $\mathbb{S}^1=K(\Z,1)=\Omega K(\Z,2)$ as algebras (by $\Omega$ we denote the based loop space).
	This also holds as algebras with anti algebra involution if
	we use for $\Omega K(\Z,2)$ the involution 
	given by the reflection of based loops together with the involution on $K(\Z,2)$ described in~\ref{actionparttwo}
	(see also~\cite[Section 4]{AF}). 
	With these definitions, the $\Z_2$-action on $\Omega K(\Z,2)$ is indeed trivial. This ensures that the identification $\mathbb{S}^1=K(\Z,1)=\Omega K(\Z,2)$ is really compatible with the $\Z_2$-structure.
	Now non-abelian Poincaré duality~\cite[Corollary 4.6]{AF} gives us
	$\int_{\mathbb{S}^1} \mathbb{S}^1 =\int_{\mathbb{S}^1} \Omega K(\Z,2)\simeq  \Map (\mathbb{S}^1 , K(\Z,2))$, 
	 and the last homotopy equivalence is in fact
	$\mathbb{S}^1 \rtimes\Z_2$-equivariant if
	$\Map (\mathbb{S}^1 , K(\Z,2))$ is equipped with the $\mathbb{S}^1 \rtimes\Z_2$-action
	described in \ref{actionpartone} and~\ref{actionparttwo}. For the $\Z_2$-action, this was just explained; for the $\mathbb{S}^1$-action, it is a property of the Poincaré duality map.
\end{proof}

\begin{theorem}\label{thmbdiff}
	There is a homotopy equivalence of topological spaces
\begin{align}
		\DH(\mathbb{S}^1)    \simeq   B \Diff (H_{1,0}) \ . 
		\end{align}
\end{theorem}

\begin{proof}
	The strategy of the proof is to explicitly compute the homotopy orbits of the $\mathbb{S}^1 \rtimes\Z_2$-action on the mapping space
	$\Map (\mathbb{S}^1 , K(\Z,2))$ from Lemma~\ref{lemmadihedralS1}. To this end, we will use that $\mathbb{S}^1$ and $K(\Z,2)$ are both realizations of 2-groupoids; in fact, $\mathbb{S}^1$ is the realization of $* \DS \Z$ (one object plus morphisms given by $\Z$), and $K(\Z,2)$ is the realization of $*\DS * \DS \Z$ (one object plus one 1-morphism and 2-morphisms given by $\Z$).
	The mapping space $\Map (\mathbb{S}^1 , K(\Z,2))$ can be described as the 2-groupoid 
	$[*\DS \Z , *\DS * \DS \Z ]$ of functors
	$\star \DS \Z\to\star \DS \star \DS \Z$, natural transformations and modifications. Without loss of generality, we can assume that
	units are preserved strictly. Up to equivalence,
	$[*\DS \Z , *\DS * \DS \Z ]$ has 
	only one object given by the constant functor $c_*$ at $*$. A natural transformation $\sigma : c_* \to c_*$ 
	consists of the unique 1-morphism $\sigma_*: c_*(*)\to c_*(*)$ together with a 2-morphism $\sigma_n : 
	c_*(n) \circ  \sigma_* \to  \sigma_* \circ c_*(n) $ for every $n\in \Z$. This 2-morphism is an element of $\Z$, and since $\sigma$ is supposed to be natural, the map
	$n\mapsto \sigma_n$
	is a group morphism $\Z\to \Z$. Hence, the 1-morphisms of $[*\DS \Z , *\DS * \DS \Z ]$ correspond each to an element in $\Z$. 
	Let $\sigma$ and $\sigma' $ be natural transformations as above. A modification $\omega : \sigma \to 
	\sigma'$ consists of a 2-morphism $\omega : \sigma_* \to \sigma'_*$ such that $ \sigma'_n \circ  \omega =
	\omega \circ \sigma_n $. This implies that 2-morphisms between $\sigma$ and $\sigma'$ only exist if $\sigma'=\sigma$, and in this case, all
	values of $\omega \in \Z $ are allowed. The composition is straightforward to compute, and we find $[*\DS \Z , *\DS * \DS \Z ] \simeq ( *\DS \Z) \times (* \DS * \DS \Z)$. As a topological space, this is of course $K(\Z,1)\times K(\Z , 2 )$.
	Next we describe the $\mathbb{S}^1\rtimes \Z_2$-action on this space:

\begin{itemize}
	\item The $\mathbb{S}^1$-action is described by
	a natural automorphism of the identity $S: \id_{( *\DS \Z) \times (* \DS * \DS \Z)} \to \id_{( *\DS \Z) \times (* \DS * \DS \Z)} $.  Concretely, 
	$S$ consists of a 1-morphism $S_{*}: *\to*$ at the only object $*\in ( *\DS \Z) \times (* \DS * \DS \Z)$ which is
	given by $0\in \Z $,
	and for any 1-morphism $n : * \to *$ a 2-morphism $S_n : n \to n$ which is given by $n$. 
	This follows directly from unpacking the bicategorical definitions for the precomposition with the natural isomorphism $s: \id_{*\DS \Z} \to \id_{*\DS \Z} $ whose value at the only object $*$ is $s_*=1\in\Z$.
		
	\item 	The $\Z_2$-action is given by the involution $( *\DS \Z) \times (* \DS * \DS \Z) \to ( *\DS \Z) \times (* \DS * \DS \Z)$  whose action on 1-morphisms is trivial and multiplies 2-morphisms by $-1$. This follows by unpacking the
	definition of the action by precomposition with the functor $*\DS \Z \to * \DS \Z $ given by $-1$ 
	on 1-morphisms combined with postcomposition by the functor $*\DS * \DS \Z \to * \DS * \DS \Z $ given by $-1$ on 2-morphisms.
\end{itemize}
	After combining this with Lemma~\ref{lemmadihedralS1}, we arrive at 
	\begin{align}\label{Eq: HQ}
		\DH(\mathbb{S}^1) \simeq (K(\Z,1)\times K(\Z, 2))_{\horb   \mathbb{S}^1 \rtimes \Z_2} \ ,   \quad \text{and hence}\quad \Omega \DH(\mathbb{S}^1) \simeq (\mathbb{S}^1\times \Z ) \rtimes (\mathbb{S}^1\rtimes \Z_2) \ , 
	\end{align} 
		where $\mathbb{S}^1\rtimes \Z_2$ acts on
	$\mathbb{S}^1\times \Z$ as follows:
	\begin{itemize}
		\item
		$\Z_2$ acts trivially on $\Z$ and by reflection on $\mathbb{S}^1$, \item and $x\in \R /\Z = \mathbb{S}^1$ acts 
		trivially on $\mathbb{S}^1$ and sends $(0,n)\in \mathbb{S}^1\times \Z $ to $ (nx , n) $. 
	\end{itemize}
	The group $(\mathbb{S}^1\times \Z ) \rtimes (\mathbb{S}^1\rtimes \Z_2)$ is isomorphic to $\Diff(H_{1,0})$ after a straightforward rewriting.
	 This concludes the proof.
\end{proof}

\spaceplease
\section{The Hochschild complex of a self-injective ribbon Grothendieck-Verdier category\label{secgv}}
In this section, we recall the notion of Grothendieck-Verdier duality from \cite{bd} and prove the main result.
We use here the conventions from \cite{cyclic} which are dual to the ones from~\cite{bd}.

\begin{definition} A \emph{Grothendieck-Verdier category} is a monoidal category $\cat{C}$ with monoidal product $\otimes$ and an object $K\in\cat{C}$ such that for all $X \in \cat{C}$ the hom functor $\cat{C}(K,X\otimes-)$ is representable (let us denote the representing object by $DX \in \cat{C}$; this means that we have  $\cat{C}(K,X\otimes-)\cong\cat{C}(DX,-)$),
		and such that the functor $\cat{C}\to\cat{C}^\opp$ sending $X$ to $DX$ is an equivalence.
	One calls $K$ the \emph{dualizing object} and $D$ the \emph{duality functor}.
\end{definition}

Building on the notion of a braiding and a balancing on a monoidal category (both notions were recalled in the introduction), we may now define:

\begin{definition}\label{defbalancedbraided} A \emph{ribbon Grothendieck-Verdier category} is a Grothen\-dieck-Verdier category whose underlying monoidal category is equipped with a braiding and a balancing such that $\theta_{DX}=D\theta_X$ for $X\in\cat{C}$. \end{definition}
We will consider ribbon Grothendieck-Verdier categories in a \emph{linear setting}. To this end, let us establish some terminology: For an algebraically closed field $k$ that we fix for the rest of the article, a \emph{finite category} is a $k$-linear abelian category with finite-dimensional morphism spaces, finitely many isomorphism classes of simple objects and enough projective objects; additionally, one requires every object to have finite length. One can now define a symmetric monoidal bicategory $\Lexf$ of finite linear categories, left exact functors and natural transformations. The monoidal product is the Deligne product. We refer e.g.\ to \cite{fss} for an overview.

\begin{definition}
	A \emph{ribbon Grothendieck-Verdier category in $\Lexf$} is an object $\cat{C}\in\Lexf$ equipped with a ribbon Grothendieck-Verdier structure on the underlying category and a lift of the ribbon Grothendieck-Verdier structure to structure \emph{inside} $\Lexf$.
\end{definition} This means in particular that the monoidal product 
will be left exact by construction. This might seem a little confusing because monoidal products are rather right exact than left exact in practice. But note that for a Grothendieck-Verdier category in $\Lexf$, the \emph{opposite category} will have a right exact monoidal product.

\begin{definition}
	A Grothendieck-Verdier category $\cat{C}$ in $\Lexf$
	is called \emph{self-injective} if it is self-injective as linear category, i.e.\ if the projective objects of $\cat{C}$ are exactly the injective ones.
\end{definition}

\begin{remark}\label{remselfinjective}
	Self-injectivity ensures that the duality functor $D:\cat{C}\to\cat{C}^\opp$ preserves projective objects. This can be seen as follows: For any $X\in\cat{C}$, the object $DX$ is always injective because $D$ is an equivalence from $\cat{C}$ to $\cat{C}^\opp$, and by self-injectivity $DX$ is projective. For finite tensor categories (which are rigid by definition), this assumption is automatically satisfied. 
\end{remark}

We are now ready to prove the main result:
	
	\begin{theorem}\label{mainthm}
		Let $\cat{C}$ be a self-injective ribbon Grothendieck-Verdier category in $\Lexf$.
		Then its duality functor and its balancing induce on the Hochschild complex $\lint^{X\in\Proj\cat{C}}\cat{C}(X,X)$
		an action of the diffeomorphism group $\Diff(H_{1,0})$ of the solid closed torus that extends the usual cyclic symmetry of the Hochschild complex. 
		\end{theorem}

	\begin{proof}
		For $n\ge 0$,
		we define the vector space
		\begin{align}\label{eqndefHn}
		H[n] := \displaystyle \bigoplus_{X_0,\dots,X_n \in \Proj\cat{C}}  \cat{C}(X_n,X_{n-1})\otimes \dots \otimes \cat{C} (X_1,X_0)  \otimes \cat{C}(X_0,X_n) 
		\end{align}
		of loops of morphisms in $\cat{C}$ running through $n+1$ projective objects. We can see $H[n]$ as chain complex concentrated in degree zero.
		Now the precomposition with the balancing on each of the different morphisms gives us a functor \begin{align}\label{eqnfunctorsHn}	H[n] : (\star \DS \Z)^{\times (n+1)} \to \Ch\end{align} 
		that we denote by $H[n]$ again, by a slight abuse of notation. In more detail, it sends the only object of $(\star \DS \Z)^{ n+1} $ to $H[n]$ and the morphism $	(\ell_0,\dots,\ell_n) \in\Z^{\times (n+1)}$ to the map
		\begin{align}
		\left( X_n \ra{f_n} X_{n-1} \to \dots X_0 \ra{f_0} X_n  \right) \mapsto \left( X_n \ra{ f_n \theta_{X_{n}}^{\ell_n} } X_{n-1} \to \dots X_0 \ra{         f_0            \theta_{X_0}^{\ell_0}} X_n  \right)
		\end{align}
		given by precomposition with the balancing. In fact, we might as well postcompose; this would give us the same result thanks to the naturality of the balancing.

		Consider now the cyclic object used to compute cyclic homology of the topological algebra $\mathbb{S}^1$, but seen now as category-valued functor $D:   \Lambda^\opp 
		\to \Cat$ sending $[n] \to (\star \DS \Z)^{\times (n+1)}$. 
		The key observation is that the functors~\eqref{eqnfunctorsHn} extend to a functor
		\begin{align}
		H: \int D \to \Ch   \label{eqnfunctorH}
		\end{align}
		out of the Grothendieck construction $\int D$ of the functor $D:   \Lambda^\opp  
		\to \Cat$ by means of the usual cyclic structure of the Hochschild complex. In order to define this functor, recall that by the universal property of the Grothendieck construction
		 it suffices to give functors $(\star \DS \Z)^{\times (n+1)} \to \Ch$ (we have these already, see~\eqref{eqnfunctorsHn}) plus natural transformations $\alpha_f$ filling the triangles
		\begin{equation}\label{gdconeqn}
		\begin{tikzcd}
		(\star \DS \Z)^{\times (n+1)}  \ar[rrd,"\text{$H[n]$}"] \ar[dd,swap,"D(f)"] &  \ar[ldd, Rightarrow, shorten <=0.75cm, shorten >=0.01cm, "\alpha_f"] & \\ 
		& & \Ch \\
		(\star \DS \Z)^{\times (m+1)} \ar[rru,swap,"\text{$H[m]$}"] & & 
		\end{tikzcd}
		\end{equation}
		for every morphism $f$ in $\Lambda^\opp$ such that these transformations respect the composition in $\Lambda^\opp$.
		It is clear how to define the needed transformation $\alpha_f$ because the vector spaces~\eqref{eqndefHn} already form a cyclic object in the standard way. 
		We only need to verify the naturality of the $\alpha_f$ which is a direct consequence of the fact that the balancing is a natural transformation.
		
		In the next step, we will make use of the \emph{paracyclic category} $\Lambda_\infty$ \cite[Appendix~B]{ns}, a contractible category with $\mathbb{S}^1$-action such that $\Lambda = \Lambda_\infty / \mathbb{S}^1$.  
		Consider now the functor $\Lambda^\opp \to \Cat$ sending $[n]$ to the action groupoid $E \Z^{\times (n+1)} :=\Z^{\times (n+1)} \DS  \Z^{\times (n+1)}$ 
		of the regular free and transitive action of $\Z^{\times (n+1)}$ on itself, i.e.\ the total space of the universal $\Z^{\times (n+1)}$-bundle. 
		Precomposition with the quotient functor $\Lambda^\opp _\infty \to \Lambda^\opp=\Lambda_\infty^\opp / \mathbb{S}^1$ 
		 yields a functor $D_\infty : \Lambda_\infty^\opp \to \Cat$. 
		The quotient functors $\Lambda^\opp _\infty \to \Lambda^\opp$ and $E \Z^{\times (n+1)} = \Z^{\times (n+1)} \DS  \Z^{\times (n+1)}\to (\star \DS  \Z)^{\times (n+1)}$ induce a functor $Q:\int D_\infty \to \int D$. 
		Since $E\Z^{\times (n+1)}  \simeq \star$, we have an equivalence $ \Lambda_\infty^\opp \ra{\simeq} \int D_\infty$. 
		
		By sending $[n]$ to $\Z^{\times (n+1)}$, we obtain a cyclic group $C$ that we may see also as paracyclic group. By realization it gives us a topological group $G$ with $BG= \int_{\mathbb{S}^1} \mathbb{S}^1$. 
		There is a $C$-action on $\int D_\infty$ induced by the action $\Z^{\times (n+1)}$-action on $E\Z^{\times (n+1)}$.
		 The $C$-action on $\int D_\infty$ and the $\mathbb{S}^1$-action on the paracyclic category $\Lambda_\infty$
		combine into an action of $C\rtimes \mathbb{S}^1$ on $\int D_\infty$ (the $\mathbb{S}^1$-action on $C$ exists because it is	 a cyclic object) such that $\left( \int D_\infty \right) / (  C\rtimes \mathbb{S}^1)=\int D$.
		This tells us that any functor $F:\int D \to \Ch$ can equivalently be described as the underlying functor $FQ:\int D_\infty\to \Ch$
	plus the $C\rtimes \mathbb{S}^1$-equivariance of this functor, where $\Ch$ is equipped with the trivial $C\rtimes \mathbb{S}^1$-action. 
	As a result, the homotopy colimit of
	$FQ$
	 over $\int D_\infty$ carries an action of $G\rtimes \mathbb{S}^1$. 
	 By \cite[Theorem~B.3]{ns} the paracyclic category comes with a homotopy final functor $\Delta^\opp \to \Lambda_\infty^\opp$. Therefore, the functor 
	 $\Delta^\opp \to \Lambda_\infty^\opp \ra{\simeq} \int D_\infty$ is also homotopy final.
	 As a result, the chain complex $\hocolim\, FQ$ is equivalent to the realization $|F|$ of the simplicial object underlying $F$.
	 This implies that $|F|$ comes with an action of $G\rtimes \mathbb{S}^1$. 
	 Now let us apply this to the functor $H:\int D \to \Ch$ from~\eqref{eqnfunctorH}. Its realization $|H|$ is the Hochschild complex of $\cat{C}$, which now comes with an action of $G\rtimes \mathbb{S}^1= \Omega \left( \left(   \int_{\mathbb{S}^1} \mathbb{S}^1   \right)_{\horb \mathbb{S}^1}\right)$. It extends the cyclic action by construction.

		 Finally, we observe  that the symmetric monoidal bicategory of linear categories comes with a homotopy coherent action of $\Z_2$ that sends a category to its opposite category.
		A self-injective ribbon Grothendieck-Verdier category $\cat{C}$ 
		does not only come
		 with a duality functor $D: \cat{C}\to\cat{C}^\opp$, but also a pivotal structure, i.e.\ an isomorphism $D^2 \cong \id_\cat{C}$, see \cite[Corollary~8.3]{bd}. This turns $\Proj\cat{C}$ into a homotopy $\Z_2$-fixed point for the just mentioned $\Z_2$-action on linear categories (this crucially uses that $D$ preserves projective objects thanks to self-injectivity, see Remark~\ref{remselfinjective}).
		The homotopy $\Z_2$-fixed point structure is through balancing preserving functors thanks to $D\theta_X = \theta_{DX}$, see Definition~\ref{defbalancedbraided}.
		The Hochschild complex of $\Proj \cat{C}$ and its opposite category $(\Proj \cat{C})^\opp$ can be canonically identified; this is true for any linear category. For this reason, $\lint^{X\in\Proj\cat{C}}\cat{C}(X,X)$ inherits, in addition to the action of   $\Omega \left( \left(   \int_{\mathbb{S}^1} \mathbb{S}^1   \right)_{\horb \mathbb{S}^1}\right)$ that we have just established, a homotopy coherent action of $\Z_2$ from the $\Z_2$-fixed point structure (the duality $D$ acts by a chain map, and the pivotal structure gives a chain homotopy between the chain map for $D^2$ and the identity --- this follows from the functoriality of the Hochschild complex). 
		
		Since the $\Z_2$-action is by linear balancing preserving functors, the $\Z_2$-action intertwines the $\Omega \left( \left(   \int_{\mathbb{S}^1} \mathbb{S}^1   \right)_{\horb \mathbb{S}^1}\right)$-action, but in a twisted way as we explain now:
		The Hochschild complex of $\Proj\cat{C}$ and $(\Proj \cat{C})^\opp$, even though they can be identified as chain complexes, can only be identified \emph{as cyclic objects} if we precompose the cyclic object for $(\Proj \cat{C})^\opp$ with the reversal functor $r:\Lambda \to \Lambda$. As a consequence, the total action on the Hochschild complex is not an action of the product of $\Omega \left( \left(   \int_{\mathbb{S}^1} \mathbb{S}^1   \right)_{\horb \mathbb{S}^1}\right)$ with $\Z_2$, but the semidirect product with $\Z_2$, where $\Z_2$-acts by reflection on both $\mathbb{S}^1$'s in
		$\Omega \left( \left(   \int_{\mathbb{S}^1} \mathbb{S}^1   \right)_{\horb \mathbb{S}^1}\right)$ that are written as subscript and trivially on the $\mathbb{S}^1$ which is not written as subscript.
		In other words, the duality plays here the same role as the anti algebra involution
		in dihedral homology.
		In summary, the Hochschild complex of $\cat{C}$ comes with an action of the topological group
		$      \Omega \left( \left(   \int_{\mathbb{S}^1} \mathbb{S}^1   \right)_{\horb \mathbb{S}^1}\right) \rtimes \Z_2 = \Omega \left( \left(   \int_{\mathbb{S}^1} \mathbb{S}^1   \right)_{\horb \mathbb{S}^1\rtimes \Z_2}\right) $. This topological group is given by the based loops of the space
	$\left(   \int_{\mathbb{S}^1} \mathbb{S}^1   \right)_{\horb \mathbb{S}^1\rtimes \Z_2}$, i.e.\ of the dihedral homology $\DH(\mathbb{S}^1)$ of $\mathbb{S}^1$. 
		By Theorem~\ref{thmbdiff} $\DH(\mathbb{S}^1) \simeq B\Diff(H_{1,0})$. As a result, the Hochschild complex of $\cat{C}$ comes with an action of
		$ \Omega \left( \left(   \int_{\mathbb{S}^1} \mathbb{S}^1   \right)_{\horb \mathbb{S}^1}\right) \rtimes \Z_2=\Omega \DH(\mathbb{S}^1)\simeq \Diff(H_{1,0})$.
		\end{proof}

	\begin{remark}\label{remframedE2}
		By one of the main results of \cite{cyclic}, ribbon Grothendieck-Verdier categories in $\Lexf$ are equivalent to \emph{cyclic} framed $E_2$-algebras in $\Lexf$. In connection with Giansiracusa's result \cite{giansiracusa} on the relation between the modular envelope of the cyclic framed $E_2$-operad and the handlebody operad, this is used in \cite{cyclic} to build systems of handlebody group representations. These results, however, are not available for the solid closed torus and make no statements about \emph{diffeomorphism} groups. For this reason, the theory of cyclic framed $E_2$-algebras in symmetric monoidal bicategories from \cite{cyclic} does not provide a shortcut to the proof of Theorem~\ref{mainthm}. In fact, the insight rather flows in the opposite direction: The  dihedral homology computation in Theorem~\ref{thmbdiff} will enable us to improve results on the modular envelope of the framed $E_2$ operad and help us characterize systems of handlebody group representations
		in~\cite{mwansular}. 
		\end{remark}

\begin{example}\label{exhopf}
	Any finite ribbon category $\cat{C}$ in the sense of \cite{egno} is in particular a self-injective ribbon Grothendieck-Verdier category. In this case, the Grothendieck-Verdier duality comes actually from rigidity.
	A source of finite ribbon categories are finite-dimensional ribbon Hopf algebras:
	If $A$ is a finite-dimensional ribbon Hopf algebra, then the category of finite-dimensional $A$-modules is a finite ribbon category \cite[XIV.6]{kassel}.
	The Hochschild complex $\lint^{X\in\Proj \cat{C}}\cat{C}(X,X)$ is equivalent to the ordinary Hochschild complex of $A$.
	As a consequence, the Hochschild complex of $A$ comes with an action of $\Diff(H_{1,0})$ extending the cyclic action.
	Let us also describe the underlying action of $\Map(H_{1,0})\cong \Z\times\Z_2$ in more detail.
	To this end, we use the Lyubashenko coend $\mathbb{F}:= \int^{X \in\cat{C}} X^\vee \otimes X \in \cat{C}$ of $\cat{C}$ \cite{lyu}.
	By
	 \cite[Theorem~3.9]{dva}, we have $\lint^{X\in\Proj\cat{C}}\cat{C}(X,X)\simeq \cat{C}(I,\mathbb{F}_\bullet) $,  where $\mathbb{F}_\bullet$ is a projective resolution of $\mathbb{F}$. 
	By recapitulating the construction given in the proof of Theorem~\ref{mainthm} we conclude that the action of $\Map(H_{1,0})\cong \Z\times\Z_2$ 
	on the Hochschild complex $\lint^{X\in\Proj\cat{C}}\cat{C}(X,X)$ of the finite ribbon category $\cat{C}$ is given
	as follows:
	\begin{pnum}
		\item The $\Z$-factor acts by postcomposition with the balancing. Equivalently, it can be described as the chain map $\cat{C}(I,\mathbb{F}_\bullet) \to \cat{C}(I,\mathbb{F}_\bullet) $ induced by the automorphism of $\mathbb{F}=\int^{X\in\cat{C}} X^\vee \otimes X$ that applies the balancing to the first or second tensor factor under the coend (both give the same result). \label{itembalancing}
		
		\item The $\Z_2$-factor acts on $\lint^{X\in\Proj\cat{C}}\cat{C}(X,X)$ by applying the duality functor to morphism spaces
		and identifying the Hochschild complexes of $\cat{C}$ and $\cat{C}^\opp$.
		This operation is (homotopy) involutive because of the pivotal structure.
		(Under the equivalence $\lint^{X\in\Proj\cat{C}}\cat{C}(X,X)\simeq \cat{C}(I,\mathbb{F}_\bullet)$, we may equivalently say
		that the generator of $\Z_2$ acts by the map induced by the inverse antipode of the Hopf algebra $\mathbb{F}$ in $\cat{C}$ \cite{lyu}. 
		This follows from the definition of the antipode in \cite[Section~2.5]{lyu} and \cite[Lemma~6.1]{mwansular}.)\label{itembalancing2}
		\end{pnum}
		In the special case where $\cat{C}$
	is given as category of modules over a ribbon Hopf algebra $A$, the coend
	$\mathbb{F}$ is given by
	the dual $A^*_\text{coadj}$ of $A$ with its coadjoint $A$-action $A \otimes A^* \to A^*$ sending  $ a \otimes \alpha$ to the linear form $b \mapsto \alpha\left(      S(a'ba'')    \right) $ \cite[Theorem~7.4.13]{kl}. 
	Here $\Delta a = a'\otimes a''$ is the Sweedler notation for the coproduct $\Delta :A\to A\otimes A$, and $S:A\to A$ is the antipode. 	The equivalence
	$\lint^{X\in\Proj\cat{C}}\cat{C}(X,X)\simeq \cat{C}(I,\mathbb{F}_\bullet) $
	is well-known in the  Hopf algebraic special case,  see e.g.~\cite[Section~2.2]{bichon}. In that case, one obtains $CH_*(A)\simeq \Hom_A(k,  {  A^*_\text{coadj}  } _\bullet  )$ for the Hochschild complex of $A$. The $\Z$-factor acts via multiplication with the ribbon element
	while the $\Z_2$-factor acts through the inverse antipode of the Hopf algebra $A_\text{coadj}^*$ in $A\text{-mod}$. 
\end{example}

	\begin{remark}[Visibility of higher homotopy groups of diffeomorphism groups on differential graded conformal blocks]\label{dmf}
Let $\cat{C}$ be a modular category, i.e.\ a finite ribbon category with non-degenerate braiding.
Then by the main result of \cite{dmf}, building on preparations in \cite{svea,dva,svea2}, the category $\cat{C}$ gives rise to a symmetric monoidal functor $\mathfrak{F}_\cat{C}:\cat{C}\text{-}\Surfc \to \Ch$ from a category of $\cat{C}$-labeled surfaces (morphisms are sewing operations and mapping classes of surfaces up to a specific central extension coming from the framing anomaly) to chain complexes. The functor satisfies an excision property phrased in terms of homotopy coends.
The chain complex assigned to the torus is the Hochschild complex of $\cat{C}$. Hence, the Hochschild complex of a modular category comes with an action of $\text{SL}(2,\Z)$. At chain level, this is given in \cite{dva} generalizing the explicit description on the Hochschild cohomology 
of a modular category coming from a ribbon factorizable Hopf algebra 
  in \cite{svea}.
Under the equivalence
$\lint^{X\in\Proj\cat{C}}\cat{C}(X,X)\simeq \cat{C}(I,\mathbb{F}_\bullet)$,
the $\text{SL}(2,\Z)$-action is determined by the automorphisms associated to $T = \begin{pmatrix} 1 & 0 \\ 1 & 1 \end{pmatrix}$ 
and $S = \begin{pmatrix} 0 & -1 \\ 1 & \phantom{-} 0 \end{pmatrix}$. 
For $T$, this is exactly the automorphism 
described in  point~\ref{itembalancing} of Example~\ref{exhopf}.
The image of $S$ is an automorphism of $ \cat{C}(I,\mathbb{F}_\bullet)$ induced by an automorphism $\mathbb{F}\to\mathbb{F}$ which by definition in \cite{svea,dva} is the rather complicated `$S$-transformation' $\mathbb{F}\to\mathbb{F}$ \cite[Section~6]{lyu}. 
We can restrict the $\text{SL}(2,\Z)$-action along the inclusion $\Z\times\Z_2\subset \text{SL}(2,\Z)$ given by~\eqref{eqnTandRmatrix}; it sends the generator of $\Z$ to $T$ and the generator of $\Z_2$ to $S^2$.
Because the $S$-transformation $\mathbb{F}\to\mathbb{F}$ squares to the inverse of the antipode of the Hopf algebra $\mathbb{F}$ \cite[Section~6, eq.~(5)]{lyu}, it follows from point~\ref{itembalancing2} in Example~\ref{exhopf} that this $\Z\times \Z_2$-action on the Hochschild complex
agrees with the one obtained via the
 section $\Map(H_{1,0})\to\Diff(H_{1,0})$ and Theorem~\ref{mainthm}.  This has the following immediate consequence:
The `handlebody part' of the mapping class group action on the Hochschild complex of a modular category from \cite{svea,dva} admits a non-trivial extension to an action of $\Diff(H_{1,0})$. The non-triviality of the extension follows for example from the fact that one of the generators of $\pi_1(\Diff(H_{1,0}))$ acts by the cyclic symmetry and hence is non-trivial.
This is, to the best of our knowledge, the first instance of a non-trivial action of higher homotopy groups of diffeomorphism groups on conformal blocks.
\end{remark}

\spaceplease
\begin{corollary}\label{cordep}
	For self-injective
	 ribbon Grothendieck-Verdier categories,
	the Hochschild complex equipped with its $\Diff(H_{1,0})$-action is a stronger invariant than the Hoch\-schild complex with its cyclic action. More precisely, there exist ribbon Grothendieck-Verdier categories with cyclically equivalent Hoch\-schild complexes which are however not equivalent as $\Diff(H_{1,0})$-modules. 
	\end{corollary}

\begin{proof}
	Let $G$ be a finite group and $\cat{C}$ the category of finite-dimensional $k$-linear modules over the Drinfeld double $D(G)$. This is a modular category whose differential graded modular functor is described in~\cite[Example~3.13]{dmf}. From the description given there and Remark~\ref{dmf}, we may deduce that the $\Z$-factor of the handlebody group $\Map(H_{1,0})\cong \Z\times\Z_2$ acts non-trivially
	(for instance, its action on $HH_0(  \cat{C}  )$ is non-trivial). 
	But now observe that $\cat{C}$, as a linear category, is equivalent to modules over the action groupoid $G\DS G$ of the adjoint action of $G$ on itself (this is because $D(G)$-modules are equivalent to Yetter-Drinfeld modules over $k[G]$, see~\cite[Theorem~IX.5.2]{kassel}). The category of finite-dimensional $G\DS G$-modules can be endowed with a symmetric monoidal product (the level-wise monoidal product
	that one can define on any category of functors to a symmetric monoidal category;
	 this is \emph{not} the monoidal product coming from the Hopf algebra structure of the Drinfeld double) and hence with a trivial balancing (that is also a ribbon structure). We denote the resulting finite ribbon category as $\cat{C}_0$.
	Then $\cat{C}=\cat{C}_0$ holds linearly which means that their Hochschild complexes agree (and, of course, so does the $\mathbb{S}^1$-action on them). But in contrast to the Hochschild complex of $\cat{C}$, the action of $\Z\subset \Map(H_{1,0})$ is trivial because the balancing is the identity.
	\end{proof}

\begin{remark}[Ansular homology]\label{remansular}
	The cyclic action on the Hochschild chain complex naturally leads to \emph{cyclic homology} by passing to homotopy orbits of the $\mathbb{S}^1$-action.
	Since on the Hochschild complex of a self-injective ribbon Grothendieck-Verdier category in $\Lexf$, the cyclic action extends to a $\Diff(H_{1,0})$-action, it is natural to consider the homotopy orbits of this action that one might call \emph{ansular homology} (\emph{ansa} is Latin for \emph{handle}). This yields an appropriate replacement of cyclic homology sensitive to the ribbon structure. The computation of these homotopy orbits lies beyond the scope of this article.
	\end{remark}

\end{document}